\documentclass{amsart}
\usepackage{bbm}
\usepackage{amsfonts}
\usepackage{mathrsfs}
\usepackage{amsthm}
\usepackage{txfonts}
\usepackage{palatino,amssymb,epsfig,amsmath,latexsym,amsthm}
\usepackage{palatino,amssymb,epsfig,amsmath}
\usepackage{latexsym ,amsthm,epsf,xypic,epic,amscd}
\usepackage{palatino, amsmath, amsthm, amssymb, amscd, mathrsfs}

\theoremstyle{plain}
\newtheorem{theorem}{Theorem}[section]
\newtheorem{lemma}[theorem]{Lemma}

\newtheorem{proposition}[theorem]{Proposition}
\newtheorem{corollary}[theorem]{Corollary}

\newtheorem{Bounded Diameter Lemma}[theorem]{Bounded Diameter Lemma}
\theoremstyle{definition}
\newtheorem{definition}[theorem]{Definition}

\newtheorem{remark}[theorem]{Remark}

\title{Intersection number and stability of some inscribable graphs}
\author{Jinsong Liu and Ze Zhou}
\address{HUA Loo-Keng Key Laboratory of Mathematics, Chinese Academic of Sciences, Beijing 100190, China }
\address{Institute of Mathematics, Academic of Mathematics $\&$ System Sciences,
Chinese Academic of Sciences, Beijing 100190, China}
\email{liujsong@math.ac.cn zhouze@amss.ac.cn}

\begin{document}

\maketitle

\begin{abstract}
A planar graph is inscribable if it is combinatorial equivalent to
the skeleton of an inscribed polyhedron in the unit sphere.
For an inscribable graph, if in its combinatorial equivalent class
we could also find a polyhedron inscribed in each convex surface
sufficiently close to the unit sphere $\mathbb{S}^2$, then we call
such an inscribable graph a stable one.

By combining the Teichm\"{u}ller theory of packings with differential topology
method, in this paper we shall investigate the stability of some inscribable graphs.


\end{abstract}

\setcounter{section}{-1}

\section{Introduction}\label{In}

A graph is called planar if it can be embedded in the unit sphere
${\mathbb S}^2$. And a planar graph is called inscribable if it can
be realized as the skeleton of the convex hull of a set of finite
points lying over the unit sphere. In the book \cite{Stein}, the
Swiss mathematician Jakob Steiner asked for a combinatorial
characterization of those inscribable graphs. To be specific, in
which cases does a  polyhedral graph (the skeleton of a
polyhedron) can be combinatorially equivalent to the skeleton of a
convex polyhedron inscribed in the sphere?

This seems to be a rather intractable problem. In fact, it's almost
a hundred years later when Steinitz \cite{Steinitz} found an example
of "non-inscribable" graph in 1927. Whereafter, more and more
non-inscribable graphs are discovered. For instance, the polyhedral
graph of the following singly-truncated cube is exactly the simplest
non-inscribable one.

\begin{figure}[htbp]\centering
\includegraphics[width=0.40\textwidth]{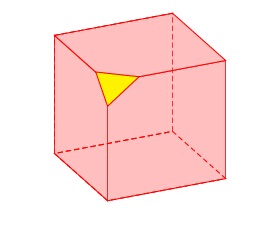}
\caption{An example of non-inscribable graph}
\end{figure}

Moreover, due to the Klein Model
of hyperbolic 3-space, one could regard an inscribed polyhedron as
an ideal hyperbolic polyhedron. In view of such an observation,
Rivin \cite{Rivin1} then completely resolved Steiner's problem by an
investigation of the geometry of ideal hyperbolic polyhedra.

For a polyhedral graph $G$, let $G^{\ast}$ denote its dual graph. We
call a set of edges $\Gamma=\{e_1,e_2,\cdots,e_k\} \subset G$ a {\it
prismatic circuit}, if the {\it dual edges}
$\{e_1^{\ast},e_2^{\ast},\cdots,e_k^{\ast}\}$ form a simple closed
curve in the dual graph $G^{\ast}$ and does not bound a face in
$G^{\ast}$. Rivin's theorem \cite{Rivin1} is then stated as follows.

 \begin{theorem}\label{Thm}
A polyhedral graph $G=G(V,E)$ is of inscribable type if and only if
there exists a weight $w$ assigned to its edges set $E$ such that:
 \begin{itemize}
\item[$(W1)$] For each edge $e \in E$, \:$0 < w(e) < 1/2$.
\item[$(W2)$] For each vertex $v$, the total weights of all edges incident to $v$ is equal to
$1$.
\item[$(W3)$] For each prismatic circuit $\gamma \subset E$,
the total weights of all edges in $\gamma$ is strictly greater than
1.
\end{itemize}
 \end{theorem}

Note that the condition $(W2)$ is equivalent that the sum of the
weights of edges bounding a face in the dual graph $G^{\ast}$ is
equal to $1$.

In addition, for any given polyhedral graph $G$, Hodgson, Rivin and
Smith \cite{Rivin2} indicate that there exists a polynomial time
algorithm (in the number of vertices) to decide whether it is
inscribable.

\medskip
These consequences are really elegant. However, a
"sphere" in the real physical world often doesn't mean a standard
sphere in mathematic sense.
It then seems significant to go a further step to
consider the stability problem of inscribable graphs. Namely, given
any convex surface $\tilde{S} \subset \mathbb R^3$ sufficiently close
to the unit sphere $\mathbb{S}^2$, for an inscribable graph $G$, is
there always a polyhedron $P_{G,\tilde{S}}$ inscribed in $\tilde{S}$
with skeleton combinatorially equivalent to $G$?

\medskip
In what follows, to formulate the above question as a mathematic one,
let's introduce some notions which will depict the exact meaning of
"sufficiently close".

Suppose that $S_1:\mathbb{\hat{C}} \,\substack{f_1 \\ \rightarrow \\
\,}\, \mathbb{R}^3$, $S_2:\mathbb{\hat{C}} \,\substack{f_2 \\
\rightarrow \\ \,}\, \mathbb{R}^3$ are two $C^k$ embeddings of the
Riemann sphere in the $3$-dimensional Euclidean space
$\mathbb{R}^3$. Given $\epsilon>0$, we say $S_1,S_2$ are
$\epsilon$-$C^k$-close to each other, if the $C^k$-norm of every
coordinate component of $f_1-f_2$ is less than $\epsilon$. For
example, if two embedding sphere are $C^3$-close to each other, it
follows from the elementary surface theory that the images of $S_1,
S_2$ and their curvatures will close to each other (see \cite{Do}).
Particularly, if
$\tilde{S}:\mathbb{\hat{C}} \,\substack{\tilde{f} \\ \rightarrow \\
\,}\, \mathbb{R}^3$ is an embedding sphere which is
$\epsilon$-$C^3$-close to the unit sphere $\mathbb{S}^2$ for some
sufficiently small $\epsilon>0$, then the surface $\tilde{S}$ is
both strictly convex and sufficiently round.

For any given inscribable graph $G$, suppose that there exists an $\epsilon>0$
such that: for any surface $\tilde{S}(\epsilon)$
which is $\epsilon$-$C^k$-close to the unit sphere $\mathbb{S}^2$, there is
always a polyhedron $P_{G,\tilde{S}(\epsilon)}$ inscribed in
$\tilde{S}(\epsilon)$ with skeleton combinatorially equivalent to
$G$. Then we say $G$ is $C^k$-stable. Recalling Rivin's result
(Theorem \ref{Thm}), the problem on how to characterize an
inscribable graph is equivalent to solve a system of linear
inequalities. However, due to the non-openness of the solutions
space of these inequalities, there may exist inscribable graph which
isn't stable. That implies the stability problem of inscribable
graphs wouldn't be a trivial task.

\medskip
Now let $P=P(\mathcal{V},\mathcal{E},\mathcal{F})\subset
\mathbb{R}^3$ be given a convex polyhedron. For every vertex $v \in
\mathcal V$, we cut a small pyramid from $P$ by a plane which is
near to $v$ and transversal to every edge $e\in \mathcal E$
emanating from $v$.  Thus we obtain a new polyhedron
$P_{\mbox{\large $\diamond$}}$, called the \textbf{truncated
polyhedron} of $P$. Denote by $G(P_{\mbox{\large $\diamond$}})$ the
skeleton of $P_{\mbox{\large $\diamond$}}$.

\begin{figure}[htbp]\centering
\includegraphics[width=0.7\textwidth]{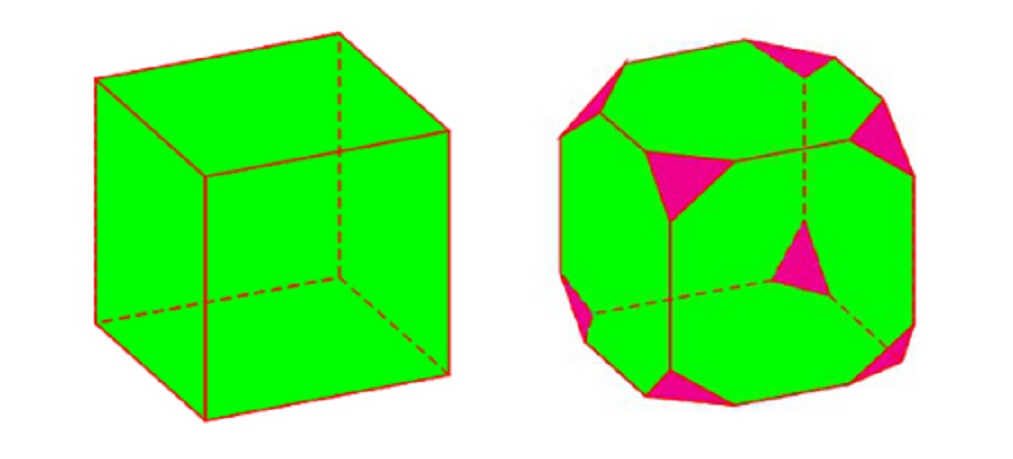}
\caption{The cube and its truncated polyhedron}
\end{figure}

In this paper we shall prove

 \begin{theorem}\label{Main0}
Let $P,P_{\mbox{\large $\diamond$}}$ and $G(P_{\mbox{\large
$\diamond$}})$ be as above. Assume that the degree $d(v)$ of each
vertex $v\in \mathcal V$ is odd. Then the graph $G(P_{\mbox{\large
$\diamond$}})$ is inscribable and $C^1$-stable.
 \end{theorem}

In addition, for a polyhedral graph $G(P)=(\mathcal{V},\mathcal{E},
\mathcal{F})$, let's construct a new graph $G_{+}(P)$ as follows. More precisely, for every edge $e\in \mathcal{E}$, we associate it with a
vertex $\mathbbm{v}_e$. Whenever two different edges $e_1,e_2\in
\mathcal E$ both belong to a common face $f\in \mathcal{F} $ and
meet at a same vertex $v\in \mathcal{V} $, we then connect an edge
from $\mathbbm{v}_{e_1}$ to $\mathbbm{v}_{e_1}$.
Thus we obtain a
new graph $G_{+}(P)$ associated to $P$, which is called the \textbf{
rectified graph} of the polyhedron $P$.

\begin{figure}[htbp]\centering
\includegraphics[width=0.5\textwidth]{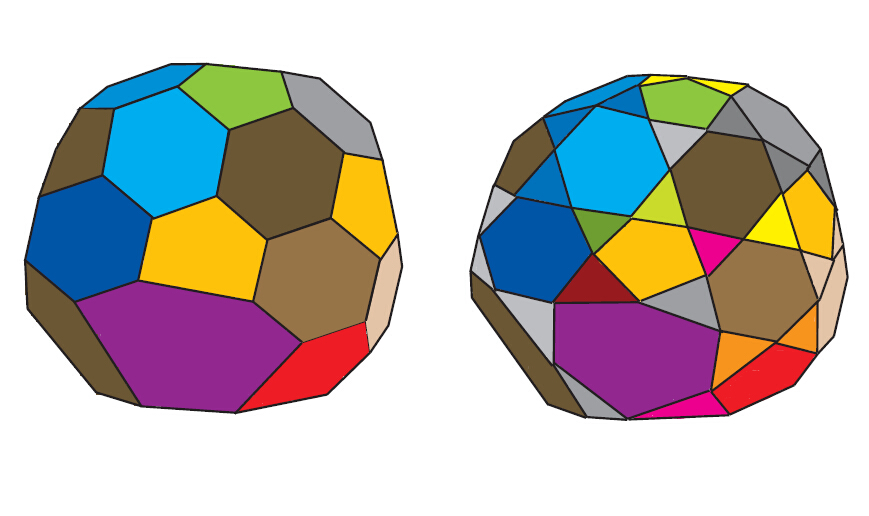}
\caption{The dodecahedron and its rectified graph}
\end{figure}

For the rectified graph
$G_{+}(P)=(\mathcal{V}_{+},\mathcal{E}_{+},\mathcal{F}_{+})$,
obviously we have
$$
|\mathcal{V}_{+}|=|\mathcal{E}|,\: |\mathcal{E}_{+}|=2|\mathcal{E}|,
\: |\mathcal{F}_{+}|=|\mathcal{V}|+|\mathcal{F}|.
$$
Furthermore, we have

\begin{theorem}\label{Main1}
Let $P,G_{+}(P)$ be as above. If $d(v)$ is odd for any vertex $v\in
\mathcal{V}$, then $G_{+}(P)$ is inscribable and $C^3$-stable.
\end{theorem}

Given a compact strictly convex surface $K$, for any affine half
space $H^+$ with $K \nsubseteq H^+$, the intersection $H^+ \cap K$
is either empty, or a point, or a topological disk. In the last case
we call it a $K$-disk, and its boundary (in $K$) a $K$-circle. We
recall that a planar graph $G$ is $K$-inscribable if there exists a
polyhedron $P_{G}$ inscribed in $K$ with skeleton combinatorially
equivalent to the graph $G$.

In terms of the above conventions, to prove Theorem \ref{Main1} is
equivalent to prove that there exists $\epsilon>0$ such that
$G_{+}(P)$ is $\tilde{S}(\epsilon)$-inscribable provided that the
embedding surface $\tilde{S}(\epsilon)$ is $\epsilon$-$C^3$-close to
the unit sphere $\mathbb{S}^2$. Recall that
$G(P)=(\mathcal{V},\mathcal{E}, \mathcal{F})$ and
$G_+(P)=(\mathcal{V}_+,\mathcal{E}_+, \mathcal{F}_+)$. To acquire
such a polyhedron, we need to find the vertices set $\mathcal V_+$
such that: (1) they correspond to the tangent points of the
$\tilde{S}(\epsilon)$-circle packing realizing the graph
$G^{\ast}(P)$, where $G^{\ast}(P)$ is the dual graph of the skeleton
of $P$; (2) if $e_1,e_2,\cdots,e_k\in \mathcal{E}$ are incident to a
same vertex $v\in \mathcal{V}$, then the corresponding points
$\mathbbm{v}_{e_1},\mathbbm{v}_{e_2},\cdots,\mathbbm{v}_{e_k} \in
\mathcal V_+$ locate in a same plane.

Hence it's necessary to prove that the intersection of these two
configuration spaces is non-empty. By combining the intersection number
theory from differential topology with a homotopy technique, we shall obtain the desired
result. Similarly, Theorem \ref{Main0} could be deduced by means of
transversality theory.

\bigskip
We now briefly describe how this paper is organized. In the
preliminary section we briefly give an introduction to
transeversality theory and intersection number theory, which will
play an important role throughout this paper. In Section 3 we study
the Teichm\"{u}ller theory of packings, which characterizes the
configuration space of $K$-circle packings. Section 4 is devoted to
the proof of Theorem \ref{Main1}. The last section provides a
geometric insight into the tangent space of another configuration.
With the help of this method, we demonstrate a tranvsersality
theorem which leads to a proof of Theorem \ref{Main0}. Furthermore,
we complete some details on the computation of intersection number
used in Section 3.

{\bf Notational Conventions}.

Through this paper, for any given set $A$ we use the notation $|A|$
to denote the cardinality of $A$.

%\begin{aligned}
%\mathbb S^2:\, \&\mathbb{\hat{C}}\rightarrow \mathbb{R}^3\\
%&z\rightarrow (\frac{Re\,z}{1+|z|^2},\frac{Im\,z}{1+|z|^2},\frac{|z|^2-1}{1+|z|^2}).\\
%\end{aligned}
%$$
\bigskip
\section{Preliminaries}\label{Pn}
In this section, we will introduce several definitions and notations from
differential topology, especially transversality and intersection
number. Please refer to \cite{Gui, Hir} for background on these
notions.

\bigskip
First of all, assume that $M,N$ are two oriented smooth manifolds, and $S\subset N$ is a submanifold.

\begin{definition}
Suppose that $f:M\rightarrow N$ is a ${C}^{1}$ map. Given $A\subset
M$, we say $f$ is transverse to $S$ along $A$, denoted by
$f\mbox{\Large $\pitchfork$}_{A}S$, if
\[
Im(df_{x})+T_{f(x)}S=T_{f(x)}N,
\]
whenever $x\in A\cap f^{-1}(S)$. When $A=M$, we simply denote
$f\mbox{\Large $\pitchfork$} S$.
\end{definition}

Let $S\subset N$ be a closed submanifold such that $dim M + dim S =
dim N$. Suppose $\Lambda \subset M$ is an open subset with compact
closure $\bar{\Lambda} \subset M$. Given a continuous map
$f:M\rightarrow N$ such that $f(\partial \Lambda)\cap S=\emptyset$,
where $\partial\Lambda=\bar{\Lambda}\setminus\Lambda$, we will
define a topological invariant $I(f,\Lambda,S)$, called the
intersection number between $f$ and $S$ in $\Lambda$.

\bigskip
If $f\in C^{0}(\bar{\Lambda},N)\cap C^{\infty}(\Lambda,N)$ such that
$ f\mbox{\Large $\pitchfork$}_{\Lambda}S$, then $\Lambda \cap
f^{-1}(S)$ consists of finite points. For each $x \in \Lambda \cap
f^{-1}(S)$, the $sgn(f,S)_{x}$ at $x$ is $+1$, if the orientations
on $Im(df_{x_{j}})$ and $T_{f(x_j)}S$ "add up" to preserve the
prescribed orientation on $N$, and $-1$ if not.

\begin{definition}
If $\Lambda \cap f^{-1}(S)=\{x_1,x_2,\cdots,x_m\}$, then we define
the intersection number between $f$ and $S$ in $\Lambda$ to be
$$
I(f,\Lambda,S):=\sum\nolimits_{j=1}^m sgn(f,S)_{x_j}.
$$
\end{definition}

The proof of the following proposition is in the same style as that
of the homotopy invariance of Brouwer degree. Please see \cite{Gui,Hir}, or Milnor's book \cite{Mil}.

\begin{proposition}\label{Prop1}
Suppose that $f_{i}\in C^{0}(\bar{\Lambda},N)\cap
C^{\infty}(\Lambda,N)$, $f_{i}\mbox{\Large $\pitchfork$}_{\Lambda}S
$ and $f_{i}(\partial \Lambda)\cap S=\emptyset, \:i=0,1$. If there
exists a homotopy
$$
H \in C^{0}({I\times \bar\Lambda},N)
$$
such that $H(0, \cdot)=f_0(\cdot), \:\: H(1, \cdot)=f_1(\cdot)$, and
 $H(I\times\partial \Lambda)\cap S=\emptyset$, then
$$
I(f_{0},\Lambda,S)=I(f_{1},\Lambda,S).
$$
\end{proposition}

The next lemma, which helps us to manipulate the intersection number for
general mappings, is a consequence of Sard's theorem \cite{Gui,Hir}.
\begin{lemma}\label{Lem1}
For any $f\in C^{0}(\bar{\Lambda},N)$ with $f(\partial \Lambda)\cap
S=\emptyset$, there exists $g\in C^{0}(\bar{\Lambda},N)\cap
C^{\infty}(\Lambda,N)$ and $H\in C^{0}(I\times\bar{\Lambda},N)$ such
that
\begin{itemize}
\item[$(1)$] $g\mbox{\Large $\pitchfork$}_{\Lambda}S$;
\item[$(2)$] $H(0,\cdot)=f(\cdot),H(1,\cdot)=g(\cdot)$;
\item[$(3)$] $H(I\times\partial \Lambda)\cap S=\emptyset$.
\end{itemize}
\end{lemma}
The above lemma, together with Proposition \ref{Prop1}, allows us to
define the intersection numbers for general continuous mappings.

\begin{definition}
For any $f\in C^{0}(\bar{\Lambda},N)$ with $f(\partial \Lambda)\cap
S=\emptyset$, we can define the intersection number
$$
I(f,\Lambda,S)=I(g,\Lambda,S).
$$
where $g$ is given as Lemma \ref{Lem1}.
\end{definition}
By Proposition \ref{Prop1}, the intersection number $I(f,\Lambda,S)$
is well-defined. Furthermore, we have the following homotopy
invariance property of this quantity.
\begin{theorem}\label{Thm1}
For $i=0,1$, suppose that $f_{i}\in C^{0}(\bar{\Lambda},N)$ such
that $f_{i}(\partial \Lambda)\cap S=\emptyset$. If there exists
$H\in C^{0}(I\times\bar{\Lambda},N)$ such that
\begin{itemize}
\item[$(1)$] $H(0,\cdot)=f_{0}(\cdot),H(1,\cdot)=f_{1}(\cdot)$,
\item[$(2)$] $H(I\times\partial \Lambda)\cap S=\emptyset$,
\end{itemize}
then we have $I(f_{0},\Lambda,S)=I(f_{1},\Lambda,S)$.
\end{theorem}

In particular, it immediately follows from the definition that:
\begin{theorem}\label{Thm3}
If $I(f,\Lambda,S)\neq 0$, then we have $\Lambda \cap
f^{-1}(S)\neq\emptyset$.
\end{theorem}

\bigskip
\section{Teichm\"{u}ller theory of packings}\label{Cp}

Given a compact strictly convex surface $K$, in this section we
shall introduce the Teichm\"{u}ller theory of $K$-circle packings
with the same contact graph.

Roughly speaking, a $K$-circle (or $K$-disk) packing $\mathcal {P}$
is a configuration of $K$-circles $\{C_{v}:v\in V\}$ (or disks
$\{D_{v}:v\in V\}$) with specified patterns of tangency. The contact
graph (or nerve) of $\mathcal{P}$ is a graph $G_{\mathcal{P}}$,
whose vertex set is $V$ and an edge appears if and only if the
corresponding $K$-circles (or $K$-disks) touch.

Given a planar graph $G=G(V,E)$, let's fix a vertex $v_{0}\in V$ and
three ordered edges $e_1,e_2,e_3 \in E$ emanating from $v_0$. We
call the $4$-tuple $\mathscr{O}=\{v_0,e_1,e_2,e_3\}$ a combinatorial
frame associated to the graph $G$. Suppose $\mathcal {P}=\{C_{v}\}$
is a $K$-circle packing with the contact graph $G_{\mathcal {P}}=G$.
Denoting by $p_1,p_2,p_3$ the three tangent points corresponding to
the edges $e_1,e_2,e_3$, we call $\mathcal {P}$ a normalized
$K$-circle packing with mark
$\mathfrak{M}=\{\mathscr{O},p_1,p_2,p_3\}$.

For the compact strictly convex surface $K$, without loss of
generality, we now assume it lies below the plane $\{(x,y,z)\in
\mathbb R^3: z=1\}$ and is tangent to this plane at the point
$N=(0,0,1)$. The point $N$ is regarded as the "North Pole" of $K$.
Let $h:K \rightarrow \mathbb C \cup \{\infty\}$ denote the
"stereographic projections" with $h(N)=\infty$. Since $h$ can be
extended to a diffeomorphism between $K$ and $\hat{\mathbb C}$, we
then endow $\partial K$ with a complex structure by pulling back the
standard complex structure of $\hat{\mathbb C}$. Hence, up to
conformal equivalence, we can identify $K$ with the Riemann sphere
$\hat{\mathbb C}$.

Given a convex polyhedron
$P=P(\mathcal{V},\mathcal{E},\mathcal{F})\subset \mathbb{R}^3$, we
recall that $G^{\ast}(P)$ is the dual graph of the skeleton of $P$.
Denote $G^*(P)=(V, E)$. Let us fix a disk packing
$\mathcal{P}_0=\{D_0(v)\}_{v \in V}$ on the unit sphere $\mathbb
S^2(\cong\hat{\mathbb{C}})$ with the contact graph $G^*(P)$. Denote
$\hat{\mathbb{C}}-\cup_{v\in V}D_{0}(v)=\{I_1,I_2,\cdots, I_m\}$.
For each component $I \in \{I_1,I_2,\cdots, I_m\}$, we call it an
open interstice. Evidently, $I$ is a topological polygon. The region
$I$ has only finitely many boundary components. And each boundary
component is a piecewise smooth curve formed by finitely many
circular arcs or circles. Each (maximal) circular arc or circle on
the boundary $\partial I$ belongs to a unique circle in the disk
packing $\mathcal P_0$, and therefore is marked by an element of
$V$. The region $I$, together with a marking of the circular arcs or
circles on its boundary by elements of $V$ is called an {\it
interstice} of $\mathcal P_0$.

For each interstice $I$ of $\mathcal P_0$, we can define a conformal
polygon as pairs $h:I\rightarrow \hat{\mathbb{C}}$, where $h$ is
a quasiconformal embedding. For details on quasiconformal mappings,
please refer to Ahlfors' book \cite{Ahl}. The conformal polygons are
considered as analogs of the conformal quadrangles.

Denote $\partial I=\{\gamma_1, \gamma_2,\cdots, \gamma_n\}$, where
$\{\gamma_j\}_{1\leq j\leq n}$ is a marking of the circular arcs or
circles on its boundary. We say two such quasiconformal embeddings
$h_1,h_2: I\rightarrow \hat{\mathbb{C}}$ are Teichm\"{u}ller
equivalent, if the composition mapping $h_2\circ (h_{1})^{-1}:
h_1(I)\rightarrow h_2(I)$ is isotopic to a conformal homeomorphism
$f: h_1(I)\rightarrow h_2(I)$ such that for each side
$\gamma_{j}\subset\partial I$, $f$ maps $h_1(\gamma_{j})$ onto
$h_2(\gamma_{j})$.

\begin{definition}
The Teichm\"{u}ller space of $I$, denoted by $\mathcal{T}_I$, is the
space of all equivalence classes of quasiconformal embeddings
$h:I\rightarrow \hat{\mathbb{C}}$.
\end{definition}
\begin{remark}\label{Rem}
If the interstice $I$ is $k-$sided, it follows from the classical
Teichm\"{u}ller theory that $\mathcal{T}_I$ is diffeomorphic to the
Euclidean space $\mathbb{R}^{k-3}$. See e.g \cite{LV}.
\end{remark}

Denote $\mathcal {T}_{G^*(P)}=\prod_{i=1}^{m}\mathcal {T}_{I_i}$,
where $\{I_1,I_2,\cdots, I_m\}$ are all interstices of the circle
packing $\mathcal P_{0}$. Due to Remark \ref{Rem}, we easily verify
that $\mathcal {T}_{G^*(P)}\cong \mathbb{R}^{2|\mathcal
E|-3|\mathcal V|}$.

\bigskip
Recall that a $K$-disk is defined as the intersection $H^+ \cap K$,
where $H^+$ is an affine half space which intersects $K$. Its
boundary is called a $K$-circle. Naturally, we call $\mathcal
{P}=\{C_{v}:v\in V\}$ a $K$-circle packing, if all $C_{v} (v\in V)$
are $K$-circles. As far as these packings concerned, Liu-Zhou
\cite{Liu-Zhou} have established the following result, which will be
used in this paper as well. It's proof is a combination of the
methods due to Schramm \cite{Schr} and Rodin-Sullivan \cite{RS}.

\begin{lemma}\label{Deform}
Let $K$, $P$, $G^*(P)$ and $\mathcal {T}_{G^*(P)}$ be as above. Suppose $p_1, p_2, p_3$ are three distinct points in $K$.
For any
 \[
 \big[\tau\big]=([\tau_{1}],[\tau_{2}],\cdots,[\tau_{m}]) \in \mathcal {T}_{G^*(P)},
 \]
there exists a unique $K$-circle packing $\mathcal {P}_{K}([\tau])$
realizing the dual graph $G^{\ast}(P)$ with mark
$\{\mathscr{O},p_1,p_2,p_3\}$. Moreover, it's interstice
corresponding to $I_i$ is endowed with the given complex structure
$[\tau_{i}]$, $\:1\leq i\leq m$.
\end{lemma}

\bigskip
\section{Proof of the main theorems}\label{Aap}

Recall that $P=P(\mathcal{V},\mathcal{E},\mathcal{F})\subset
\mathbb{R}^3$ is the given convex polyhedron and $P_{\mbox{\large
$\diamond$}}$ is the corresponding truncated polyhedron. To prove the main
theorems, in this section we will construct, step by step, two
configurations spaces $Z_{oc}$, $Z(P_{\mbox{\large $\diamond$}})$
associated to $P$, $P_{\mbox{\large $\diamond$}}$ respectively.

In view of analytic geometry, we know that each affine half space
$H^{+}\subset\mathbb{R}^3$ can be defined as
\[
H^{+}=\{(x,y,u): Ax+By+Cu+D \geq0\}\,\quad (A^2+B^2+C^2\neq0).
\]
Hence each $H^{+}$ is uniquely determined by the exterior unit
normal vector and the intercept. In other words, it's could be
depicted by a point in $\mathbb{S}^2\times \mathbb{R}$.

Let $Z_{\mathcal F}$ denote the space $(\mathbb{S}^2\times
\mathbb{R})^{|\mathcal F|}$. Namely, a point $z_{\mathcal F}\in
Z_{\mathcal F}$ gives a choice of an affine half space (or an
oriented plane) for each $\mbox{\small $f$}\in \mathcal F$.
$Z_{\mathcal F}$ will be called the $\mathcal F$-configuration
space, and a point $z_{\mathcal F}\in Z_{\mathcal F}$ will be called
a $\mathcal F$-configuration. For each $\mathcal F$-configuration
$z_{\mathcal F}\in Z_{\mathcal F}$, we denote by $z_{\mathcal F}(f)$
the oriented plane corresponding to the face $\mbox{\small $f$}\in
\mathcal F$.

For any $e\in \mathcal E$, there are $f_1,f_2\in \mathcal F$ such
that $f_1\cap f_2=e$. Let $Z_{\mathcal Fe}\subset Z_{\mathcal F}$ be
the set of $\mathcal F$-configurations $z_{\mathcal F}$ such that
$z_{\mathcal F}(f_1)$ is parallel to $z_{\mathcal F}(f_2)$.
Moreover, let $Z_{\mathcal FR}\subset Z_{\mathcal F}$ be the set of
$\mathcal F$-configurations $z_{\mathcal F}$ such that the
intersection
$$
z_{\mathcal F}(f_{i_1})\cap z_{\mathcal F}(f_{i_2}) \cap z_{\mathcal F}(f_{i_3})
$$
contains more than one points for at least one triple
$\{i_1,i_2,i_3\} \subset \{1,2,\cdots, |\mathcal F|\}$. Evidently,
both $Z_{\mathcal Fe}$ and $Z_{\mathcal FR}$ are closed in
$Z_{\mathcal F}$, which implies that
\[
Z_{\mathcal FO}=Z_{\mathcal F}\setminus\Big((\cup_{e\in \mathcal E}Z_{\mathcal Fe})\cup Z_{\mathcal FR}\Big)
\]
is open in $Z_{\mathcal F}$. Hence it's a manifold with the same dimension as
$Z_{\mathcal F}$.

\bigskip
Let $Z$ denote the space $Z_{\mathcal FO}\times
\mathbb{R}^{|\mathcal E|}\times \mathbb{R}^{|\mathcal E|}$. Namely,
a point $z\in Z$ gives a choice of a half space(or an oriented
plane) for each $\mbox{\small $f$}\in \mathcal F$, and a choice of
two points corresponding to the vertices $v_1, v_2 \in \mathcal V$
in the line $z_{\mathcal F}(f_1)\cap z_{\mathcal F}(f_2)$, where
$f_1\cap f_2=v_1v_2=e\in\mathcal E$. Similarly, we call $Z$ the
configuration space. In addition, a point $z\in Z $ will be called a
configuration.

For a configuration $z\in Z$, here and hereafter we simply denote by
$z(f)$ the oriented plane corresponding to the face $\mbox{\small
$f$}\in \mathcal F$. Moreover, if $f_1\cap f_2=e \in \mathcal E$,
then we denote by $z(ve)$ the point in $z(f_1)\cap z(f_2)$
corresponding to the vertex $v \in \mathcal V$.

Now let $Z_{oc}\subset Z$ denote the set of configurations $z$ such
that $z(ve_i),z(ve_j),z(ve_k)$ are not collinear whenever
$e_i,e_j,e_k$ are three distinct edges incident to the same vertex
$v\in \mathcal V$. Obviously, $Z_{oc}$ is open in $Z$. Hence,
$Z_{oc}$ is a manifold with the same dimension as $Z$. More
precisely,
\[
dim Z_{oc}= dim Z=3|\mathcal F|+2|\mathcal E|.
\]

For any $v\in \mathcal V$, suppose that $e_1,e_2,\cdots,e_{d(v)}$
are all edges of $P$ incident to the vertex $v$, where $d(v)$ is the
degree of $v$. Denote by $Z_{v}\subset Z_{oc}$ the set of
configurations $z$ such that $z(ve_1),z(ve_2),\cdots, z(ve_{d(v)})$
belong to the same plane. Define
$$
Z(P_{\mbox{\large $\diamond$}})=\cap_{v\in \mathcal V}Z_{v}.
$$
In some cases, a configuration $z\in Z(P_{\mbox{\large
$\diamond$}})$ would correspond to a polyhedron in $\mathbb{R}^3$
combinatorially equivalent to $P_{\mbox{\large $\diamond$}}$.
However, it's worth pointing out that there do exist configurations
corresponding to other intricate geometric patterns as well. Aside
from these complexity, we have:

\begin{lemma}\label{Lem3}
$Z(P_{\mbox{\large $\diamond$}})$ is a closed submanifold of
$Z_{oc}$ with dimension $dim Z(P_{\mbox{\large
$\diamond$}})=3|\mathcal E|+6$.
\end{lemma}

\begin{proof}
As above, let $e_1,e_2,\cdots,e_{d(v)}$ be the edges of the
polyhedron $P$ emanating from $v$. For each $i=1,2,\cdots, d(v)$,
denote $z(ve_i)=(x_i,y_i,u_i)$.

Consider the matrix
$$
\begin{pmatrix}
 x_1&y_1&u_1&1\\
 x_2&y_2&u_2&1\\
 x_3&y_3&u_3&1\\
 \vdots&\vdots&\vdots&\vdots\\
 x_{d(v)}&y_{d(v)}&u_{d(v)}&1
\end{pmatrix}
$$
Then $z(ve_1),z(ve_2),\cdots, z(ve_{d(v)})$ belong to the
same plane if and only if the rank of the above matrix is less than
$4$. Equivalently, the determinant
 $$
 R(z(ve_{i_1}),z(ve_{i_2}), z(ve_{i_3}), z(ve_{i_4}))=
 \left|\begin{array}{cccc}
 x_{i_1}&y_{i_1}&u_{i_1}&1\\
 x_{i_2}&y_{i_2}&u_{i_2}&1\\
 x_{i_3}&y_{i_3}&u_{i_3}&1\\
 x_{i_4}&y_{i_4}&u_{i_4}&1
 \end{array}\right|=0.
 $$
for each subset $\{i_1,i_2,i_3,i_4\} \subset \{1,2,\cdots, d(v)\}$.

In view of the definition of the space $Z_{oc}$, it follows that
$z(ve_{j_1}),z(ve_{j_2}),z(ve_{j_3})$ aren't collinear for any three
different subscripts $\{j_1,j_2,j_3\}\subset \{1,2,\cdots, d(v)\}$.
That implies that $0$ is the regular value of the smooth function
$R(ve_{i_1},ve_{i_2},ve_{i_3},ve_{i_4})$. Owing to the regular value
theorem \cite{Hir}, $Z(P_{\mbox{\large $\diamond$}})$ is then a
closed submanifold of $Z_{oc}$. Moreover, we have
$$
dim Z(P_{\mbox{\large $\diamond$}})=3|\mathcal F|+2|\mathcal E|-(
  \sum\nolimits_{v\in \mathcal V}
 d(v)-3)=3|\mathcal F|+2|\mathcal E|-(2|\mathcal E|-3|\mathcal V|)=3|\mathcal E|+6,
$$
where the last identity comes from Euler's formula.
\end{proof}

Let $K$ be a given compact strictly convex surface. Choose a
combinatorial frame $\mathscr{O}$ for $G^{\ast}(P)$ and three
different points $p_1,p_2,p_3$ in $K$. For each $[\tau]\in \mathcal
{T}_{G^{\ast}(P)}$, from Lemma \ref{Deform}, it follows that there
is a unique normalized $K$-circle packing $\mathcal{P}_{K}([\tau])$
realizing the graph $G^{\ast}(P)$ with the mark
$\mathfrak{M}=\{\mathscr{O},p_1,p_2,p_3\}$.

Note that $G^{\ast}(P)=(V,E,F)$ is the dual graph of the polyhedral
graph $G(P)=(\mathcal V, \mathcal E, \mathcal F)$. For any $f \in
\mathcal F$, then $f^{\ast} \in V$. For the $K$-circle packing
$\mathcal{P}_{K}([\tau])$, denote by $H^{+}(f^\ast)$ the oriented
plane corresponding to the vertex $f^\ast \in V$.

If $e \in \mathcal E$, then $e^{\ast} \in E$. In addition, let
$p_{e^{\ast}}\in K$ be the tangent points associating with the edge
$e^{\ast} \in E$. We now associate $\mathcal{P}_{K}([\tau])$ with a
configuration $z(\tau)\in Z_{oc}$ such that
$z(\tau)(f)=H^{+}(f^\ast)$ and
$z(\tau)(v_{1}e)=z(\tau)(v_{2}e)=p_{e^{\ast}}$. Consequently, it
gives rise to the following mapping:
\[
f_{K,\mathfrak{M}}:\mathcal {T}_{G^{\ast}(P)}\longrightarrow
Z_{oc}\hookrightarrow Z.
\]

Furthermore, an elementary calculation gives
$$
\begin{aligned}
&dim Z(P_{\mbox{\large $\diamond$}})=3|\mathcal F|+2|\mathcal E|-(2|\mathcal E|-3|\mathcal V|)=3|\mathcal E|+6,\\
&dim \mathcal {T}_{G^{\ast}(P)}=2|\mathcal E|-3|\mathcal V|,\\
&dim \mathcal {T}_{G^{\ast}(P)}+dim Z(P_{\mbox{\large $\diamond$}})=3|\mathcal F|+2|\mathcal E|=dim Z_{oc}.\\
\end{aligned}
$$

These identities remind us of the intersection number theory. In
order to apply this tool, it's necessary to find a proper compact
set $\Lambda\subset \mathcal{T}_{G^{\ast}(P)}$ such that
$f_{K,\mathfrak{M}}(\partial \Lambda)\cap Z(P_{\mbox{\large
$\diamond$}})=\emptyset$.

Given $\epsilon >0$, we denote by $\mathcal
{B}(\mathbb{S}^2,\epsilon)$ the set of compact convex surfaces which
are $\epsilon$-$C^3$-close to the unit sphere $\mathbb{S}^2$.

\begin{lemma}\label{Lem2}
Assume that $d(v)$ is odd for every $v\in \mathcal V$. Then there
exists $\epsilon >0$ such that: Any $K\in \mathcal
{B}(\mathbb{S}^2,\epsilon)$ is convex and there is a compact set
$\Lambda \subset \mathcal {T}_{G^{\ast}(P)}$ such that
$f_{K,\mathfrak{M}}(\partial \Lambda)\cap Z(P_{\mbox{\large
$\diamond$}})=\emptyset$.
\end{lemma}

\begin{proof}
Due to the continuity, the above lemma will be deduced if we could
prove the existence of $\Lambda$ such that
$f_{K,\mathfrak{M}}(\partial \Lambda)\cap Z(P_{\mbox{\large
$\diamond$}})=\emptyset$ for $K=\mathbb{S}^2$.

To simplify notations, let $f_0=f_{\mathbb{S}^2,\mathfrak{M}}$. Note
that a configuration $z \in f_0({T}_{G^{\ast}(P)})\cap
Z(P_{\mbox{\large $\diamond$}})$ corresponds to an ideal polyhedron
with skeleton combinatorially equivalent to $G_{+}(P)$. We could
consider this ideal polyhedron as a circle packing $\mathcal{P}_{0}$
on the Riemann sphere realizing $G^{\ast}(P)$. For any $v\in
\mathcal V$, we assume that $e_1,e_2,\cdots,e_{d(v)}$ are the edges
of the ideal polyhedron emanating from $v$. Let
$p_{e_{1}^{\ast}},p_{e_{2}^{\ast}},\cdots,p_{e_{d(v)}^{\ast}}$ be
the corresponding tangent points of the packing $\mathcal{P}_{0}$.
It follows that
$p_{e_{1}^{\ast}},p_{e_{2}^{\ast}},\cdots,p_{e_{d(v)}^{\ast}}$ are
contained in a common plane, which implies that they are contained
in a common circle $C_v$.

\begin{figure}[htbp]\centering
\includegraphics[width=0.63\textwidth]{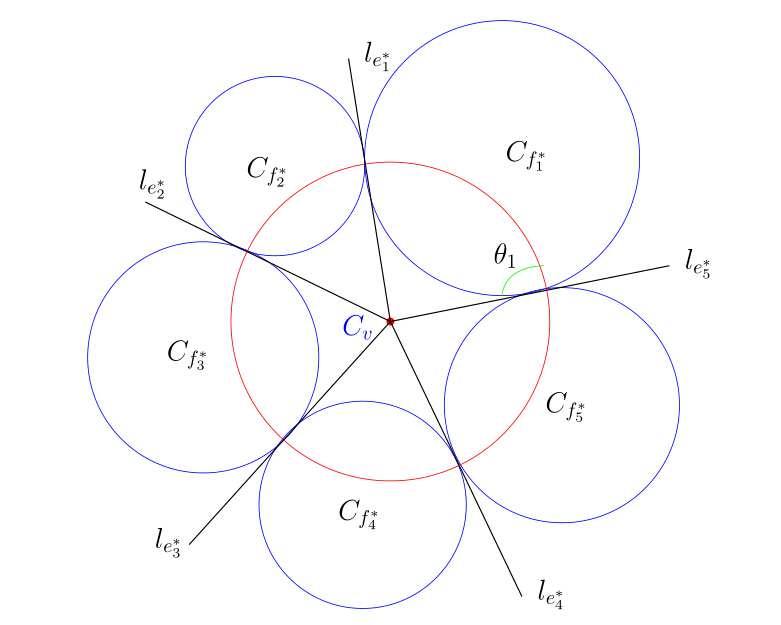}
\caption{}
\end{figure}

For $i\in \{1,2,\cdots,d(v)\}$, let $\theta_i$ be the dihedral angle
between $C_v$ and $C_{f_{i}^{\ast}}$, where $C_{f_{i}^{\ast}}\in
\mathcal{P}_{0}$ is the circle that contains the tangent points
$p_{e_{i}^{\ast}},p_{e_{i+1}^{\ast}}$. Since $d(v)$ is odd, a simple
computation then shows that

\begin{equation}\label{equation}
\theta_1=\theta_2=\cdots=\theta_{d(v)}=\frac{\pi}{2}.
\end{equation}

This implies that the lines
$l_{e_{1}^{\ast}},l_{e_{2}^{\ast}},\cdots,l_{e_{d(v)}^{\ast}}$
intersect at the center of the circle $C_v$, where
$l_{e_{1}^{\ast}}, 1\leq i\leq d(v)$, is the common tangent line of
the circles $\{C_{f_{i}^{\ast}}, C_{f_{i+1}^{\ast}}\}$ in the
packing ${\mathcal P}_0$.

Now we assume, by contradiction, that there is not such a compact
set $\Lambda$. Then there is a sequence of $[\tau]_{n}\in
f_{0}^{-1}(Z(P_{\mbox{\large $\diamond$}}))$ such that the
corresponding normalized packings $\mathcal {P}_{n}=\mathcal
P(\tau_{n})$ satisfy one of the follow two possibilities:

\begin{itemize}
\item{}As $n\rightarrow \infty$, there exists $f^{\ast} \in V$, such that the corresponding circles
$\{{C}_{n,f^{\ast}}\}$ in the packings $\mathcal{P}_n$ tends to a
point;
\item{} For some $v^{\ast}$, as $n\rightarrow \infty$, the distance of two non-adjacent arcs of the interstice $I_{v^{\ast},n}$
of the packings $\mathcal{P}_n$ tends to zero.
\end{itemize}

In the first case, suppose that there exists at least one circle
tending to a point. Note that any three circles with disjoint
interiors can not meet at a common point. Therefore, all circles in
the packing sequence $\mathcal{P}_n$ will degenerate to points,
except for at most two circles. It contradicts to our normalization
conditions. We thus rule out the first possibility.

Now we turn to the second case. Noting that $[\tau]_{n}\in
f_{0}^{-1}(Z(P_{\mbox{\large $\diamond$}}))$, they correspond to a
sequence of ideal polyhedra $P_{n}$. Hence the tangent lines of the
packings $\mathcal{P}_n$ will separate the non-adjacent arcs. On the
other hand, we have known that the sizes of all circles in
$\mathcal{P}_n$ have positive infimum. These facts tell us that the
distance of such non-adjacent arcs can't tend to zero, which rules
out the second possibility.
\end{proof}

\begin{remark}
It's worth pointing out that Equation \eqref{equation} wouldn't hold
any more if $d(v)$ is even for some $v \in \mathcal V$. In fact,
this seems to be the main obstruction on why we couldn't extend
Theorem \ref{Main0} and Theorem \ref{Main1} to more general cases.
\end{remark}

Assume that $K\in \mathcal {B}(\mathbb{S}^2,\epsilon)$. If we could
prove $I(f_{K,\mathfrak{M}},\Lambda,Z(P_{\mbox{\large
$\diamond$}}))\neq 0$, then Theorem \ref{Thm3} implies that
$f^{-1}_{K,\mathfrak{M}}(Z(P_{\mbox{\large $\diamond$}}))\cap
\Lambda \neq \emptyset$, which proves Theorem \ref{Main1}. Recalling
Theorem \ref{Thm1}, to determine the intersection numbers, let's use
a homotopy method.

Note that $K$ is a given compact strict convex surface. Without
loss of generality, we assume that its diameter is larger than $1$.
Furthermore, assume that the unit sphere $\mathbb{S}^{2}$ is
internally tangent to $K$ at the point $N=(0,0,1)$. Then $N=(0,0,1)$
could be considered as the common "North Pole" of $\mathbb{S}^2$ and
$K$.

\bigskip
Let $h_0$, $h_1$ be the "stereographic projections" for
$\mathbb{S}^{2}$, $K$ respectively.
Define a one parameter family of closed surfaces by
\[
\{s \cdot h_1^{-1}(z)+(1-s) \cdot h_0^{-1}(z): z \in \hat{\mathbb
C}\}.
\]
For each $s \in [0,1]$, the above set is a compact strictly convex
surface in $\mathbb R^3$. Denote it by $K_s$ Then $\{K_s\}_{1\leq s
\leq 1}$ is a family of compact strictly convex surface joining
$\mathbb{S}^2$ and $K$. Similarly, we can endow the smooth convex
surface $K_s$ with the complex structure $\hat{\mathbb{C}}$ for each
$s \in [0,1]$ by the "stereographic projection". Moreover, with the
help of Lemma \ref{Deform}, we could construct a mapping
\[
 f_{s}=f_{K_s,\mathfrak{M}}:\mathcal{T}_{G^{\ast}(P)}\rightarrow Z_{oc},
\]
which is a homotopy from $f_{0}$ to $f_{K,\mathfrak{M}}$. Furthermore, if $K\in
\mathcal {B}(\mathbb{S}^2,\epsilon)$, from Lemma \ref{Lem2} it
follows that there exists $\Lambda \subset
\mathcal{T}_{G^{\ast}(P)}$ such that $f_s(\partial \Lambda)\cap
Z(P_{\mbox{\large $\diamond$}})=\emptyset$ for all $s \in [0,1] $.
Furthermore, we conclude that:

\begin{theorem}\label{Thm2}
Suppose that $d(v)$ is odd for each $v\in \mathcal V$.
Given any $K\in \mathcal {B}(\mathbb{S}^2,\epsilon)$, then $I(f_{K,\mathfrak{M}},\Lambda,Z(P_{\mbox{\large $\diamond$}}))=1$.
\end{theorem}
\begin{proof}

Due to Theorem \ref{Thm1}, it's necessary to calculate
$I(f_{0},\Lambda,Z(P_{\mbox{\large $\diamond$}}))$. From the
following Proposition \ref{Prop2}, we have
$I(f_{0},\Lambda,Z(P_{\mbox{\large $\diamond$}}))=1$. It thus
completes the proof.
\end{proof}

\begin{proposition}\label{Prop2}
Suppose that $d(v)$ is odd for any $v\in \mathcal V$.
Then $I(f_{0},\Lambda,Z(P_{\mbox{\large $\diamond$}}))=1$
\end{proposition}
The proof of this result is postponed to the next section.

\medskip
Up to now, we have developed all the necessary results for our
purpose. It's ready to prove one of the main results of this paper.

\begin{proof}[\textbf{Proof of Theorem \ref{Main1}}]
As pointed out, it is an immediate consequence of Theorem \ref{Thm3}
and Theorem \ref{Thm2}, which completes the proof.
\end{proof}

\bigskip
\section{Transversality and computation of intersection numbers}\label{PT}

It remains to prove Theorem \ref{Main0} and Proposition \ref{Prop2}.
To reach this goal, we shall make use of transversality theory and a
hopotopy method.

Let's employ a consequence concerning the Teichm\"{u}ller theory of
circle patterns. Recall that $G^{\ast}(P)$ is the dual graph of the
skeleton of $P$. In \cite{H-L}, He-Liu have proved the following
theorem:

\begin{lemma}\label{Lem4}
Suppose that a weight function $w: E\rightarrow [0,\pi/2]$
satisfies the following two conditions:
\begin{itemize}
\item[$(i)$]If three distinct edges $e_i^{\ast},e_j^{\ast},e_k^{\ast}$
form a simple closed loop in $G^{\ast}(P)$, then $w(e_i^{\ast})+w(e_j^{\ast})+w(e_k^{\ast})<\pi$.
\item[$(ii)$]If four distinct edges $e_i^{\ast},e_j^{\ast},e_k^{\ast},e_l^{\ast}$
form a simple closed loop in $G^{\ast}(P)$, then $w(e_i^{\ast})+w(e_j^{\ast})+w(e_k^{\ast})+w(e_l^{\ast})<2\pi$.
\end{itemize}
For any
\[
\big[\tau\big]=\big([\tau_{{1}}],[\tau_{{1}}],\cdots,[\tau_{{n}}]\big)
\in {\mathcal T}_{G^{\ast}(P)},
\]
there exists a unique normalized circle pattern $\mathcal P_w(\big[\tau\big])$
with contact graph $G^{\ast}(P)$ and dihedral angle
$w(e^\ast):e^\ast\in E$. Moreover, the corresponding interstices of
$\mathcal P_w(\big[\tau\big])$ are endowed with the given complex structure
$[\tau_{i}], 1\leq i\leq n$.
\end{lemma}

Let $W$ be the set of weight functions that satisfy the above
conditions $(i)$ and $(ii)$. Lemma \ref{Lem4} implies that we can
define, for each $w\in W$, a mapping $f_{w,\mathfrak{M}}:\mathcal
{T}_{G^{\ast}(P)}\rightarrow Z_{oc}$ via associating every
$[\tau]\in \mathcal{T}_{G^{\ast}(P)}$ with the unique normalize
circle pattern realizing the complex structure $[\tau]$. More
precisely, we define $f_{w,\mathfrak{M}}([\tau])=z$, where $z$ is
the unique configuration such that: (1) $z(f)$ (we view it as an
oriented plane) contains the circle $C_{f^{\ast}}$; (2) $z(v_1e)$,
$z(v_2e)$ are the two intersection points of $C_{f_1^{\ast}},
C_{f_2^{\ast}}$ corresponding the vertices $v_1, v_2 \in \mathcal
V$, where $f_1\cap f_2=e$.

Denoting $w_0=(0,0,\cdots,0)$ and $w_s=sw+(1-s)w_0, \: s\in [0,1]$,
then $f_{w_s,\mathfrak{M}}$ is a homotopy from
$f_0=f_{w_0,\mathfrak{M}}$ to $f_{w,\mathfrak{M}}$. Furthermore,
suppose that we have chosen $w\in W$ sufficiently close to $w_0$. By
using a similar argument as in Lemma \ref{Lem2}, we deduce that
there exists a compact subset $\Lambda \subset \mathcal
{T}_{G^{\ast}(P)}$ such that $f_{w_s,\mathfrak{M}}(\partial
\Lambda)\cap Z(P_{\mbox{\large $\diamond$}})=\emptyset$ for all
$s\in [0,1]$.

\bigskip
In order to calculate $I(f_{w,\mathfrak{M}},\Lambda,Z(P_{\mbox{\large
$\diamond$}}))$, it seems necessary to investigate the
transversality between $f_{w,\mathfrak{M}}$ and $Z(P_{\mbox{\large $\diamond$}})$.
We thus need the follwoing Andreev's theorem \cite{And1,And2,Do},
which provide us a geometric insight into the tangent space
of $Z(P_{\mbox{\large $\diamond$}})$. Denote by $\mathcal
E_{\diamond}$ the edges set of $P_{\mbox{\large $\diamond$}}$. Then
we have

\begin{lemma}\label{Lem5}
Let $P_{\mbox{\large $\diamond$}}$ be a trivalent polyhedron in
$\mathbb{R}^3$ with a weight function
$w_\diamond:E_\diamond\rightarrow(0,\pi/2]$ attached to its edge
set. There is a compact hyperbolic polyhedra $Q_{\mbox{\large
$\diamond$}}$ combinatorially equivalent to $P_{\mbox{\large
$\diamond$}}$ with the dihedral angle $\theta(e_\diamond)$ equal to
$w(e_\diamond)$ if and only if the following conditions hold:
\begin{itemize}
\item[$(1_a)$]If three distinct edges $e_{\diamond i},e_{\diamond j},e_{\diamond k}$
meet at a vertex, then $w(e_{\diamond i})+w(e_{\diamond j})+w(e_{\diamond k})>\pi$.
\item[$(2)$]If $\{e_{\diamond i},e_{\diamond j},e_{\diamond k}\}$ is a prismatic 3-circuit,
then $w(e_{\diamond i})+w(e_{\diamond j})+w(e_{\diamond k})<\pi$.
\item[$(3)$]If $\{e_{\diamond i},e_{\diamond j},e_{\diamond k},e_{\diamond l}\}$ is
a prismatic 4-circuit, then $w(e_{\diamond i})+w(e_{\diamond j})+w(e_{\diamond k})+w(e_{\diamond l})<2\pi$.
\end{itemize}
Furthermore, this polyhedron is unique up to isometries of
$\mathbb{B}^3$.
\end{lemma}

\bigskip
Recall that the Klein model of the hyperbolic 3-space $\mathbb H^3$.
In this model, $\mathbb H^3$ is identified with the interior of the
unit ball $\mathbb B^3 \subset \mathbb R^3\subset \mathbb {RP}^3$.
In addition, a geodesic in $\mathbb H^3$ corresponds to the
intersection of a straight line of $\mathbb{R}\mathbb{P}^{3}$ with
$\mathbb{B}^ {3}$, and a totally geodesic plane in $\mathbb H^3$
corresponds to the intersection of a plane of
$\mathbb{R}\mathbb{P}^{3}$ with $\mathbb{B}^ {3}$. Then a convex
body in $\mathbb H^3$ is represented by a convex body in $\mathbb
B^3$.

Furthermore, in terms of Bao-Banahon \cite{Bao}, a hyperideal
polyhedron $Q_{hi}$ is defined to be a compact convex polyhedron in
$ \mathbb{R}\mathbb{P}^{3}$ whose vertices locate outside of the
closed unit ball $\mathbb B^3$ and whose edges all meet $\mathbb
B^3$.

Observe that the truncated polyhedron $P_{\mbox{\large $\diamond$}}$
is a a trivalent polyhedra if and only if $G^{\ast}(P_{\mbox{\large
$\diamond$}})$ is a triangular graph, where
$G^{\ast}(P_{\mbox{\large $\diamond$}})$ is the dual graph of the
skeleton of the polyhedron $P_{\mbox{\large $\diamond$}}$. Recall
the definition of prismatic circuits given in Section \ref{In}.

By either Circle Pattern Theorem \cite{Mar,Thu} or Hyperideal
Polyhedra Theorem \cite{Bao}, we have:

\begin{lemma}\label{Lem6}
Let $P_{\mbox{\large $\diamond$}}$ be a trivalent polyhedron in
$\mathbb{R}^3$ with a weight function $w_\diamond:\mathcal
E_\diamond\rightarrow[0,\pi/2]$ attached to its edges set. There is
a compact hyperideal polyhedra $Q_{hi}$ combinatorially equivalent
to $P_{\mbox{\large $\diamond$}}$ with the dihedral angle of
$e_\diamond$ equal to $w(e_\diamond)$ if and only if the following
conditions hold:
\begin{itemize}
\item[$(1_b)$]If three distinct edges $e_{\diamond i},e_{\diamond j},e_{\diamond k}$
meet at a vertex, then $w(e_{\diamond i})+w(e_{\diamond
j})+w(e_{\diamond k})\leq\pi$.
\item[$(2)$]If $\{e_{\diamond i},e_{\diamond j},e_{\diamond k}\}$ is a prismatic 3-circuit,
then $w(e_{\diamond i})+w(e_{\diamond j})+w(e_{\diamond k})<\pi$.
\item[$(3)$]If $\{e_{\diamond i},e_{\diamond j},e_{\diamond k},e_{\diamond l}\}$ is a prismatic
$4$-circuit, then $w(e_{\diamond i})+w(e_{\diamond j})+w(e_{\diamond k})+w(e_{\diamond l})<2\pi$.
\end{itemize}
This polyhedron is unique up to an element of $PO(3, 1)$, where the
group $PO(3, 1)$ consists of those projective transformations of
$\mathbb{RP}^3$ which respect the unit sphere $\mathbb S^2\subset
\mathbb R^3\subset \mathbb{RP}^3$.

Furthermore, a vertex is located on the unit sphere if and only if
the equality holds in Condition $1_b$ for this vertex.
\end{lemma}

\bigskip
Recall that the skeleton of the truncated polyhedron
$G(P_{\diamond})=(\mathcal{V}_{\diamond},\mathcal{E}_{\diamond},
\mathcal{F}_{\diamond})$. We call $e_{\diamond} \in \mathcal
E_{\diamond}$ an ordinary edge if $e_{\diamond}$ actually
corresponds to an edge $e\in\mathcal E$ in the polyhedron $P$. Other
edges of $\mathcal E_{\diamond}\backslash \mathcal E$ are called
special edges. Without leading to ambiguity, here and hereafter we
shall not distinguish an ordinary edge $e\in \mathcal E$ with its
corresponding edge in $\mathcal E_{\diamond}$.

Similarly, we can define the ordinary faces and the special faces of
$\mathcal{F}_{\diamond}$. Obviously, each special face of
$\mathcal{F}_{\diamond}$ corresponds to a vertex of $\mathcal V$.

\vspace{0.1cm} By using the above two lemmas, we have the following
result.

\begin{proposition}\label{Prop3}
Suppose that $d(v)$ is odd for any $v\in \mathcal V$. Then $f_{w,\mathfrak{M}}\mbox{\Large$\pitchfork$}Z(P_{\mbox{\large $\diamond$}})$.
\end{proposition}

\begin{proof}
If $z=f_{w,\mathfrak{M}}([\tau]) \in f_{w,\mathfrak{M}}(\mathcal
{T}_{G^{\ast}(P)})\cap Z(P_{\mbox{\large $\diamond$}})$, then the
configuration $z$ is combinatorially equivalent to the truncated
polyhedron $P_{\Diamond}$.

From Theorem \ref{Thm}, Lemma \ref{Lem5} and Lemma \ref{Lem6}, it
follows that there exists an injection
$$
\Psi:PO(3, 1)\times U\rightarrow Z(P_{\mbox{\large $\diamond$}}),
$$
where $U$ is the relatively open convex set of
$(0,\pi/2]^{3|\mathcal E|}$ defined by the constraint conditions
$(2)$ and $(3)$. Moreover, an elementary computation shows that the
map $\Psi$ is differentiable. Note that

\[
dim PO(3, 1)+dim U=6+3|\mathcal E|=dim Z(P_{\mbox{\large
$\diamond$}}).
\]
The injectivity then tells us that there exist
$(m_1,w_{\diamond1})\in PO(3, 1)\times U$ such that
$z=\Psi(m_1,w_{\diamond1})$ and the pushing map
 \[
 \Psi_{\ast}:T_{m_1}PO(3, 1)\times T_{w_{\diamond1}}U \rightarrow T_{z}Z(P_{\mbox{\large $\diamond$}})
 \]
is a linear isomorphism.

For any ordinary edge $e$, denote by $v_1e,v_2e$ the two end points
of the edge $e$ in the truncated polyhedral (corresponding to the
vertices $v_1, v_2 \in \mathcal V$). Moreover, for $i=1,2$, we
define the defect curvature $k(v_ie)$ at the vertex $v_ie$ to be
\[
k(v_{i,e})=\pi-(w(e)+w(e_{v_i,1})+w(e_{v_i,2})),
\]
where $e,e_{v_i,1},e_{v_i,2}$ are the three distinct edges incident
to the vertex $v_ie$ in the truncated polyhedron $P_{\mbox{\large
$\diamond$}}$. Note that the tangent space $T_{w_{\diamond1}}U$ is
expanded by the vectors
\[
\Bigg\{\frac{\partial}{\partial
w(e_{\diamond1})},\,\frac{\partial}{\partial w(e_{\diamond2})},
 \cdots,\frac{\partial}{\partial w(e_{\diamond 3|\mathcal E|})}\Bigg\}.
\]
When $d(v)$ is odd for each $v\in \mathcal V$, it's not hard to
deduce that this tangent space is equivalent to the $\mathbb
R-$linear space expanded by
\[
\Bigg\{\frac{\partial}{\partial w(e_1)}, \frac{\partial}{\partial
k(v_{1}e_1)},\frac{\partial}{\partial
k(v_{2}e_1)},\cdots,\frac{\partial}{\partial w(e_{|\mathcal E|})},
\frac{\partial}{\partial k(v_{1}e_{|\mathcal
E|})},\frac{\partial}{\partial k(v_{2}e_{|\mathcal E|})}\Bigg\},
\]
where $\{e_1,e_2,\cdots, e_{|\mathcal E|}\}$ are all ordinary edges
of the polyhedron $P_{\mbox{\large $\diamond$}}$.

Since $\Psi_{\ast}:T_{m_1}PO(3, 1)\times T_{w_{\diamond1}}U
\rightarrow T_{z}Z(P_{\mbox{\large $\diamond$}})$ is a linear
isomorphism, we can identify $PO(3,1)$ with the space of all marks
$\mathfrak{M}=\{\mathscr{O},p_1,p_2,p_3\}$.

\bigskip
Any special face of $\mathcal F_{\Diamond}$ corresponding to the
vertex $v \in \mathcal V$ is a $d(v)$-sided polygon. We fix $d(v)-3$
diagonal lines emanating from the same vertex of this special face.
Using these
$$
\sum_{v \in \mathcal
V}(d(v)-3)=2|\mathcal E|-3|\mathcal V|
$$
diagonal lines serving as bending lines, from Theorem \ref{Thm} it
follows that there is a family of ideal hyperbolic convex polyhedra
$z(\theta_1, \theta_2,\cdots, \theta_{2|\mathcal E|-3|\mathcal
V|})$, where $0\leq \theta_j \leq \epsilon_0, 1\leq j \leq
2|\mathcal E|-3|\mathcal V|$, are exterior dihedral angles. Please
refer to Part I \& II of the book \cite{CEM}. It implies that the
tangent map
$$
df_{w,\mathfrak{M}}: T_{[\tau]}\mathcal {T}_{G^{\ast}(P)} \rightarrow
T_{z}Z_{oc}
$$
is differentiable and injective.

Now we assume that $\textbf{t}\in
df_{w,\mathfrak{M}}(T_{[\tau]}\mathcal {T}_{G^{\ast}(P)})\cap
T_{z}Z(P_{\mbox{\large $\diamond$}})$ is a tangent vector. Then
$\textbf{t}\in T_{z}Z(P_{\mbox{\large $\diamond$}})$ corresponds to
an infinitesimal change of the dihedral angle $w(e_j)$ of some
ordinary edge $e_j$, or the defect curvature $k(v_ie_j)$ of some
vertex $v_ie_j\in \mathcal {V}_\diamond$, or the mark
$\mathfrak{M}=\{\mathscr{O},p_1,p_2,p_3\}$. On the other hand,
$\textbf{t}\in df_{w,\mathfrak{M}}(T_{[\tau]}\mathcal
{T}_{G^{\ast}(P)})$, the mark and the dihedral angles of the
ordinary edges never change. Furthermore, due to the definition, the
vertices of $z(\theta_1, \theta_2,\cdots, \theta_{2|\mathcal
E|-3|\mathcal V|})$ keep locating on $\mathbb{S}^2$. Note that a
non-trivial change on defect curvature $k(v_i e_j)$ means a
deviation from $\partial \mathbb{B}^3=\mathbb{S}^2$. Hence
$\textbf{t}=0$, which thus completes the proof.
\end{proof}

\begin{corollary}
$I(f_{0},\Lambda,Z(P_{\mbox{\large $\diamond$}}))=I(f_{w,\mathfrak{M}},\Lambda,Z(P_{\mbox{\large $\diamond$}}))=1$
\end{corollary}
\begin{proof}
Owing to the rigidity of ideal hyperbolic polyhedra \cite{Rivin1},
there exists only one point in the intersection
$f_{w,\mathfrak{M}}(\mathcal {T}_{G^{\ast}(P)})\cap
Z(P_{\mbox{\large $\diamond$}})$. By Theorem \ref{Prop3}, we thus
show that $I(f_{w,\mathfrak{M}},\Lambda,Z(P_{\mbox{\large
$\diamond$}}))=1$. In view of Theorem \ref{Thm1}, the corollary is
complete.
\end{proof}

In the end, let's prove Theorem \ref{Main0}, which is another main
result of this paper.
\begin{proof}[\textbf{Proof of Theorem \ref{Main0}}]
Without loss of generality, we assume that the unit sphere $\mathbb
S^2$ is contained in the interior of the convex surface $K$. We
shall construct a new mapping $f_{w,K,\mathfrak{M}}:\mathcal
{T}_{G^{\ast}(P)}\rightarrow Z_{oc}$ associated to the data $w$ and
$K$.

As mentioned above, for any $[\tau]\in {T}_{G^{\ast}(P)}$, from
Lemma \ref{Lem4}, it follows that there exists a unique normalize
circle pattern $\mathcal{P}(w,[\tau])$ on $\mathbb{S}^2$ realizes
the data $w$ and $[\tau]$ with the contact graph $G^{\ast}(P)$.
Denote $\mathcal{P}(w,[\tau])=\{C_{f^\ast}: f\in \mathcal F\}$ and
let $H^{+}_f$ be the oriented plane where the circle $C_{f^\ast}$
locates. Now let's define $f_{w,K,\mathfrak{M}}([\tau])=z$ such that
$z(f)=H^{+}_f$ and $z(v_{1,e}),z(v_{2,e})$ are exactly the two
intersection points in $H^{+}_{f_1}\cap H^{+}_{f_2}\cap K$ when
$f_1\cap f_2=e$. We thus construct the mapping
$f_{w,K,\mathfrak{M}}:\mathcal {T}_{G^{\ast}(P)}\rightarrow Z_{oc}$.
Note that $f_{w,\mathbb{S}^2,\mathfrak{M}}=f_{w,\mathfrak{M}}$ and
the condition
$f_{w,\mathfrak{M}}\mbox{\Large$\pitchfork$}Z(P_{\mbox{\large
$\diamond$}})$ is stable under a slight $C^1$-perturbation, we prove
the statement of this theorem.
\end{proof}

\end{document}